\documentclass[11pt,centertags,oneside]{article}
\usepackage{amsmath,amstext,amsthm,amscd,typearea}
\usepackage{amssymb}
\usepackage{a4wide}
\usepackage[mathscr]{eucal}
\usepackage{mathrsfs}
\usepackage{typearea}
\usepackage{charter}
\usepackage{pdfsync}
\usepackage{url}

\usepackage{xcolor}

\usepackage[a4paper,width=15.2cm,top=4cm,bottom=4cm]{geometry}

\numberwithin{equation}{section}

\newtheorem{theorem}{Theorem}[section]
\newtheorem{definition}[theorem]{Definition}
\newtheorem{proposition}[theorem]{Proposition}

\newtheorem{lemma}[theorem]{Lemma}
\newtheorem{remark}[theorem]{Remark}

\newtheorem{conjecture}[theorem]{Conjecture}

\newcommand{\cali}[1]{\mathscr{#1}}

\newcommand{\ddc}{dd^c}

\newcommand{\codim}{{\rm codim\ \!}}

\newcommand{\Cc}{\cali{C}}
\newcommand{\Dc}{\cali{D}}

\newcommand{\Hc}{\cali{H}}

\newcommand{\Oc}{\cali{O}}

\newcommand{\Uc}{\cali{U}}
\newcommand{\Vc}{\cali{V}}

\newcommand{\C}{\mathbb{C}}

\newcommand{\N}{\mathbb{N}}

\newcommand{\R}{\mathbb{R}}
\renewcommand\P{\mathbb{P}}

\newcommand{\wed}{\wedge}

\title{\bf  Equidistribution for non-pluripolar currents on compact K\"ahler manifolds}
\providecommand{\keywords}[1]{\textbf{\textit{Keywords:}} #1}
\providecommand{\subject}[1]{\textbf{\textit{Mathematics Subject Classification 2010:}} #1}

\author{Taeyong Ahn and Duc-Viet Vu}

\newcommand{\Addresses}{{
		\bigskip
		\footnotesize
		
		\noindent\textsc{Duc-Viet Vu, University of Cologne, Division of Mathematics, Department of Mathematics and Computer Science, Weyertal 86-90, 50931, K\"oln,  Germany.}
		
		\par\nopagebreak
		\noindent
		\textit{E-mail address}: \texttt{dvu@uni-koeln.de}
\newline
		
		\noindent\textsc{Taeyong Ahn, Department of Mathematics Education, Inha University, 100 Inha-ro, Michuhol-gu, Incheon 22212, Republic of Korea.}
		
		\par\nopagebreak
		\noindent
		\textit{E-mail address}: \texttt{t.ahn@inha.ac.kr}	
}}

\date{\today}
\begin{document}
\maketitle

\begin{abstract} Let $X$ be a compact K\"ahler manifold of complex dimension $k\ge 2$ and $f: X \to X$ a surjective holomorphic endomorphism of simple action on cohomology. We prove that the sequence of normalized pull-backs of a non-pluripolar current under iterates of $f$ converges to the Green current associated with $f$.
\end{abstract}
\noindent
\keywords Green currents, equidistribution, super-potentials of currents, non-pluripolar product.
\\

\noindent
\subject{32U15}, {32Q15}, {37F80}, {37D20}.


\section{Introduction}


Let $(X, \omega)$ be a compact K\"ahler manifold of complex dimension $k\ge 2$. Let $f:X\to X$ be a surjective holomorphic endomorphism of $X$. In complex dynamics, one is interested in the asymptotic behaviour of backward or forward orbits of analytic objects under iterates of $f$. These objects could be  points, analytic sets or more generally closed positive currents. While the dynamics of endomorphisms of complex projective spaces has been much studied, the corresponding theory for endomorphisms of compact K\"ahler manifolds has not yet been fully developed.

The aim of this paper is to address one of these {}{forementioned} problems by studying the backward orbits of closed {}{positive} currents under iterates of $f$. It is reasonable that one must impose some conditions on $f$ in order to have more concrete ideas what is going on. To go into details, let us first recall the notion of dynamical degrees.  

The dynamical degree $d_s$ of order $s$ of $f$ is the spectral radius of the pull-back operator $f^*$ acting on the Hodge cohomology group $H^{s, s}(X, \C)$ for $0\le s\le k$. It is known that $d_s$ itself is an eigenvalue of $f^*$ on $H^{s, s}(X, \C)$.  It is well-known  that the function $s\to \log d_s$ is concave for $0\le s\le k$ (\cite{DS06}, \cite{Gromov}). In particular, there are integers $p$ and $p'$ with $0\le p\le p'\le k$ such that
\begin{displaymath}
	d_0<\cdots<d_p=\cdots=d_{p'}>\cdots>d_k.
\end{displaymath}
It is known that $d_0=1$ and $d_k=\#f^{-1}(x)$ for generic $x\in X$. We call $d_k$ the topological degree of $f$ and $d_p$ the main dynamical degree of $f$. Note that $d_s$ is always an eigenvalue of $f^*$ on $H^{s,s}(X, \R)$.
We say that the action of $f$ on cohomology is \emph{simple} if $p=p'$, and $d_p$ is a simple eigenvalue and is the only eigenvalue of maximal modulus of $f^*$ on $H^{p, p}(X, \C)$.  As readers could quickly notice, the notion of simple action on cohomology is a cohomological version of the hyperbolicity of real dynamical systems. It is expected that such maps should have interesting dynamical properties. Much is known when $f$ is an automorphism with simple action on cohomology by Dinh-Sibony \cite{DS10,DS-green,DS-exponential} and Dinh-de Th\'elin \cite{DD}. One of the difficulties when passing from automorphisms to endomorphisms is that the inverse map  is multi-valued in general. 

In \cite{DNV}, it was proved that if $f$ is a holomorphic endomorphism with simple action on cohomology, then the limits $T^+_p:=\lim_{n\to\infty}d_p^{-n}(f^n)^*(\omega^p)$ and $T^-_p:=\lim_{n\to\infty}d_p^{-n}(f^n)_*(\omega^{k-p})$ exist in the sense of currents. These currents are called main dynamical Green currents. Our main theorem is as follows:

\begin{theorem}\label{thm:1st-main}
	Let $f:X\to X$ be a surjective holomorphic endomorphism of $X$. Assume that the action of $f$ on cohomology is simple with the main dynamical degree $d_p$. Let $R_1, \cdots, R_p$ be closed positive $(1, 1)$-currents. Then,
	\begin{align*}
		d_p^{-n}(f^n)^* \langle R_1\wed \cdots\wed R_p\rangle \to cT^+_p,\quad c=\dfrac{\{T^-_p\}\smile \{\langle R_1\wed \cdots\wed R_p\rangle\}}{ \{T^-_p\}\smile \{\omega^{p}\} }
	\end{align*}
	in the sense of currents where the bracket $\langle \cdot \rangle$ means the non-pluripolar product of currents, the bracket $\{\cdot\}$ means the de Rham cohomology class, to which the given closed current or form belongs.
\end{theorem}
When $X$ is projective space, it was proved for the non-pluripolar product of Cegrell class with an integrability condition for all bidegrees in \cite{AN}. Hence, the non-pluripolar product in \cite{AN} should have zero Lelong number everywhere. We note here that the current $\langle R_1\wed \cdots\wed R_p\rangle $ may have strictly positive Lelong number along an analytic subset of dimension $k-p-1$ in $X$. For example, consider $X= \P^k$ with homogeneous coordinates $z= [z_0:z_1: \cdots:z_k]$, and $R_j:= \ddc \log (|z_0|^2+ \cdots+ |z_p|^2)$ for $1 \le j \le p$. We see that the intersection $R:=R_1 \wedge \cdots \wedge R_p$ is classically well-defined because the unbounded locus of $R_j$'s is $V:=\{z_0= \cdots =z_p=0\}$ of dimension $k-p-1$. Moreover $R$ has no mass on $V$ and is smooth outside $V$. Hence $R= \langle R_1 \wedge \cdots \wedge R_p \rangle$ which has positive Lelong number along $V$ (\cite{Demailly_ag}). 

If $f$ is an automorphism with simple action on cohomology, then Dinh-de Th\'elin proved a much stronger equidistribution toward the Green current for every ``compact piece'' of a (locally) closed positive $(p,p)$-current. We refer to  \cite[Corollary 4.2]{DD} for a precise statement. As far as we {}{understand}, the proof of this result uses the fact that $f$ is an automoprhism in a crucial way.

Dinh-Sibony \cite{DS08} and Taflin \cite{T} showed that for every non-invertible self-map $f$ on the projective space $\P^k$ if quasi-potentials of a closed positive $(1,1)$-current $R$ of unit mass are locally integrable along the maximally invariant set $E$ of $f$ and $\lambda$ is  the algebraic degree of $f$, then $\lambda^{-n} (f^n)^* S$ converges to the Green current $T$ of $f$ ; see also \cite{Ahn16} for more information. We also note that  Guedj \cite{G} proved that if  $R$ a closed positive $(1, 1)$-current having zero Lelong number everywhere, then we also get the equidistribution $\lambda^{-n} (f^n)^* S \to T$. It is a good time to recall the following conjecture by Dinh-Sibony in \cite{DS08}:
\begin{conjecture}[Conjecture 1.4 in ~\cite{DS08}]\label{conj:dinh-sibony} Let $0\le s\le k$ be an integer.
	Let $f:\P^k\to\P^k$ be a holomorphic endomorphism of degree $d\geq 2$ and $T$ its Green $(1, 1)$-current. Then $d^{-sn}(f^n)^*[H]$ converges to $c_HT^s$ for every analytic subset $H$ of $\P^k$ of pure dimension $k-s$ and of degree $c_H$ which is generic. Here, $H$ is generic if either $H\cap \mathcal{E}=\emptyset$ or $\codim H\cap \mathcal{E}=s+\codim\, \mathcal{E}$ for any irreducible component $\mathcal{E}$ of every totally invariant analytic subset of $\P^k$ and $[H]$ denotes the current of integration on $H$.
\end{conjecture}

This conjecture has been solved only for $s=1$ in \cite{DS08,T}, for $s=k$ in \cite{DS10-1} and for $1<s<k$, partially { (in the case of $H\cap\mathcal{E}=\emptyset$)  in \cite{Ahn16} as cited above}. Theorem \ref{thm:1st-main} provides information in this vein of research.
\\

\noindent
\textbf{Acknowledgments.}  The research of T. Ahn was supported by the National Research Foundation of Korea
(NRF) grant funded by the Korea government (MSIT) (No. RS-2023-00250685).
The research of D.-V. Vu is partially funded by the Deutsche Forschungsgemeinschaft (DFG, German Research Foundation)-Projektnummer 500055552 and by the ANR-DFG grant QuaSiDy, grant no ANR-21-CE40-0016.

\section{Superpotentials}

In this section, we briefly recall the definition of superpotentials on a compact K\"ahler manifold and related properties that we will use. For the detail of superpotentials and continuous/bounded superpotentials, see \cite{DS10}. In particular, for projective space, see \cite{DS09}.

Let $(X, \omega)$ be a compact K\"ahler manifold of complex dimension $k\ge 2$. Let $1\le s\le k$ be an integer. For simplicity, we define the mass of a positive or negative current $S$ of bidegree $(s, s)$ to be $\|S\|:=|\langle S, \omega^{k-s}\rangle|$. Let $\Cc_s$ denote the space of closed positive $(s, s)$-currents on $X$ and $\Dc_s$ the real vector space spanned by $\Cc_s$. Define the $*$-norm on $\Dc_s$ by $\|R\|_*:=\min(\|R^+\|+\|R^-\|)$, where $R^\pm$ are currents in $\Cc_s$ such that $R=R^+-R^-$. The topology on $\Dc_s$ is as follows: a sequence $(R_n)$ in $\Dc_s$ converges to a current $R\in\Dc_s$ in $\Dc_s$ if $R_n\to R$ in the sense of currents and if $\|R_n\|_*$ is bounded independently of $n$. We will sometimes call this topology the $*$-topology. Let $\widetilde{\Dc}_s$ the subspace of smooth ones in $\Dc_s$. Let $\Dc_s^0$ and $\widetilde{\Dc}^0_s$ denote the subspaces of currents in $\Dc_s$ and $\widetilde{\Dc_s}$ whose cohomology class is $0$, respectively. Observe that on any $*$-bounded subset of $\Dc_s$, the $*$-topology is just the subspace topology inherited from the classical weak topology on the space of currents.
\medskip

\begin{definition}
	Let $1\le s\le k$ be an integer. Let $S_0$ be a current in $\Dc^0_s$. The superpotential $\Uc_{S_0}$ of $S_0$ is the following function defined for smooth currents $R_0$ in $\widetilde{\Dc}^0_{k-s+1}$ by
	\begin{displaymath}
		\Uc_{S_0}(R_0):=\langle S_0, U_{R_0}\rangle=\langle U_{S_0}, R_0\rangle
	\end{displaymath}
	where $U_{R_0}$ is a smooth $(k-s, k-s)$-current $U_{R_0}$ such that $dd^c U_{R_0}=R_0$, and $U_{S_0}$ is a $(s-1, s-1)$-current such that $dd^c U_{S_0}=S_0$.
\end{definition}

In this article, for a current $S\in \Dc_s$, a superpotential of $S$ means the superpotential of $S-\alpha_S$ where $\alpha_S$ is a smooth closed $(s, s)$-form which belongs to the same cohomology class as $S$. The properties that we will discuss about the superpotential of $S$ are independent of the choice of $\alpha_S$. Otherwise, we will clearly specify it.

\begin{definition}
	We say that $S_0\in\Dc_{s}^0$ has a continuous superpotential if $\Uc_{S_0}$ can be extended to a function on $\Dc^0_{k-s+1}$ which is continuous with respect to the $*$-topology. In this case, the extension is also denoted by $\Uc_{S_0}$ and is also called the superpotential of $S_0$. The current $S_0$ is said to have a bounded superpotential if there exists a constant $M>0$ such that for every $R_0\in\widetilde{\Dc}^0_{k-s+1}$,
	\begin{align*}
		|\Uc_{S_0}(R_0)|\le M\|R_0\|_*.
	\end{align*}
\end{definition}

Now, we consider the operators $f^*$ and $f_*$ on the space of $\Dc_s$ for $0\le s\le k$. Because of \cite[Lemma 4.7]{DNV}, the properties for $f^*$ proved in \cite{DS10} works for $f_*$ as well, and also the other way around in our case.
\begin{lemma}[{\cite[Lemma 4.7]{DNV}}]\label{lem:covering}
	Let $f$ be a surjective holomorphic map from $X$ to $X$. Then $f^{-1}$ is a holomorphic correspondence.
\end{lemma}
In general, it may not be true because $f^{-1}$ may fail to be a holomorphic correspondence for a holomorphic correspondence $f$.
\medskip

Below are some properties that we will use later.
\begin{proposition}[See {\cite[Proposition 4.1]{DNV}}]
	Let $f$ be a surjective holomorphic map from $X$ to $X$. Let $0\le s\le k$ be an integer. Then, the operators $f^*$ and $f_*$, which are well-defined on $\widetilde{\Dc}_s$, extend continuously to linear operators from $\Dc_s$ to $\Dc_s$. Moreover, they preserve the subspace $\Dc_s^0$ and the cone $\Cc_s$. For $T\in\Dc_s$, we have $f^*\{T\}=\{f^*(T)\}$ and $f_*\{T\}=\{f_*(T)\}$. We also have $(f^n)^*=(f^*)^n$ and $(f^n)_*=(f_*)^n$ on $\Dc_s$ and $H^{s, s}(X, \R)$.
\end{proposition}

Let $\Cc_s^c$ denote the cone of closed positive $(p, p)$-currents with continuous superpotentials and $\Dc_s^c$ the vector subspace of $\Dc_s$ spanned by $\Cc_s^c$. We define $\Dc_s^{0, c}$ to be the intersection $\Dc_s^{0, c}:=\Dc_s^0\cap \Dc_s^c$.
\begin{proposition}
	[See {\cite[Proposition 4.4]{DNV}}]\label{prop:img_conti_suppot} The operators $f^*$ and $f_*$ preserve the subspaces $\Dc_s^c$ and $\Dc_s^{0, c}$, and the cone $\Cc_s^c$. Let $S_0\in\Dc_s^{0, c}$ be a current, and $\Uc_{S_0}$ and $\Uc_{f^*(S_0)}$ the superpotentials of $S_0$ and $f^*(S_0)$, respectively. Then, $\Uc_{f^*(S_0)}(R)=\Uc_{S_0}(f_*(R))$ and $\Uc_{f_*(S_0)}(R)=\Uc_{S_0}(f^*(R))$ for $R\in\Dc_{k-s+1}^0$.
\end{proposition}

\begin{remark}
	[See {\cite[Remark 4.5]{DNV}}] Let $S_0\in \Dc_s^0$ and $R\in \Dc_{k-s+1}^0$ be currents, and $\Uc_{S_0}$ and $\Uc_R$ their superpotentials respectively. When $R$ is smooth, we have $\Uc_{S_0}(R)=\Uc_R(S_0)$. Therefore, we can extend $\Uc_{S_0}$ to $\Dc_{k-s+1}^{0, c}$ by setting $\Uc_{S_0}(R):=\Uc_R(S_0)$. (One may also give more precise arguments using the notion of the SP-convergence in \cite{DS10}.) Then, from the above proposition, we have $\Uc_{f^*(S_0)}(R)=\Uc_{S_0}(f_*(R))$ and $\Uc_{f_*(S_0)}(R)=\Uc_{S_0}(f^*(R))$ for $R\in \Dc_{k-s+1}^{0, c}$.
\end{remark}

It is not difficult to deduce the following proposition from \cite[Lemma 2.4.3]{DS10} and \cite[Proposition 3.2.8]{DS10} together with the above remark.
\begin{proposition}\label{prop:ext_to_conti}
	Let $S_0$ be a current in $\Dc^0_s$ with a bounded superpotential $\Uc_{S_0}$ such that there exists $M>0$ satisfying
	\begin{align*}
		|\Uc_{S_0}(R_0)|\le M\|R_0\|_*
	\end{align*}
	for every $R_0\in\widetilde{\Dc}^0_{k-s+1}$. Then, there exists a constant $c>0$ independent of $S_0$ such that
	\begin{align*}
		|\Uc_{S_0}(R')|\le cM\|R'\|_*\quad \textrm{ for }R'\in \Dc^{0,c}_{k-s+1}.
	\end{align*}
	In particular, we have
	\begin{align*}
		|\Uc_{f^*(S_0)}(R')|\le cM\|f_*(R')\|_*\quad \textrm{ for }R'\in \Dc^{0,c}_{k-s+1}
	\end{align*}
\end{proposition}

\section{Simple action on cohomology}\label{sec:simple_action}

This section summarizes some properties about simple action on cohomology from \cite{DNV}. 
\smallskip

Let $(X, \omega)$ be a compact K\"ahler manifold of complex dimension $k\ge 2$. Let $f:X\to X$ be a surjective holomorphic endomorphism of $X$ with a simple action on cohomologies. For $0\le s\le k$, we denote by $d_s$ the dynamical degree of order $s$, that is, the spectral radius of the pull-back operator $f^*$ acting on the Hodge cohomology group $H^{s, s}(X, \C)$. 

Below are properties of $f$  taken from \cite{DNV}.
\begin{lemma}
	[{\cite[Lemma 5.2]{DNV}}]\label{lem:decomposition_cohom} The sequence $d_p^{-n}(f^n)^*\{\omega^p\}$ converges in $H^{p, p}(X, \R)$ to a non-zero class $c^+$. Let $L^+$ denote the real line spanned by $c^+$. Then, we can write $H^{p, p}(X, \R)=L^+\oplus H^+$, where $H^+$ is a hyperplane of $H^{p, p}(X, \R)$ which is invariant under $f^*$. Moreover, the decomposition $H^{p, p}(X, \R)=L^+\oplus H^+$ is independent of $\omega^p$, we have $d_p^{-1}f^*(c^+)=c^+$ and the spectral radius of $d_p^{-1}f^*$ on $H^+$ is strictly smaller than $1$.
\end{lemma}

\begin{lemma}
	[{\cite[Lemma 5.3]{DNV}}] the sequence $d_p^{-n}(f^n)_*\{\omega^{k-p}\}$ converges in $H^{k-p, k-p}(X, \R)$ to a non-zero class $c^-$. Let $L^-$ be the real line spanned by $c^-$. Then, we can write $H^{k-p, k-p}(X, \R)$ $=L^-\oplus H^-$, where $H^-$ is a hyperplane of $H^{k-p, k-p}(X, \R)$ , invariant by $f_*$. Moreover, we have $d_p^{-1}f_*(c^-)=c^-$ and the spectral radius of $d_p^{-1}f_*$ on $H^-$ is strictly smaller than $1$.
\end{lemma}

\begin{proposition}[{\cite[Proposition 5.5]{DNV}}]
	\label{prop:construction_Green} The sequence $d_p^{-n}(f^n)^*(\omega^p)$ converges weakly to a closed positive $(p, p)$-current $T_p^+$ in the cohomology class $c^+$, as $n$ tends to infinity. The sequence $d_p^{-n}(f^n)_*(\omega^{k-p})$ converges weakly to a closed positive $(k-p, k-p)$-current $T_p^-$ in the cohomology class $c^-$, as $n$ tends to infinity. Moreover, $T_p^+$ and $T_p^-$ have continuous superpotentials.
\end{proposition}

We adapt some notation from \cite{DNV} for later use. Let $\Hc^-$ be a real vector space spanned by a finite number of real smooth closed $(k-p, k-p)$-forms such that the map $\alpha\to\{\alpha\}$ is a bijection from $\Hc^-$ onto the hyperplane $H^-$ in $H^{k-p, k-p}(X, \R)$. Let $\alpha_0$ be a form in the class $c^-$. Notice that $\{\omega^{k-p}\}= c^-+ c_0$ where $c_0$ is a class in $H^-$. So, we can choose $\Hc^-$ so that the form $\Omega_0:=\omega^{k-p}-\alpha_0$ belongs to $\Hc^-$. Note that $\{\Omega_0\}=c_0$.
\medskip

By the definition of $\Hc^-$, for each $c\in H^-$, there exists a unique form $\beta_c\in\Hc^-$ such that $\{\beta_c\}=c$. Note that the class of $d_p^{-1}f_*(\beta_c)-\beta_{d_p^{-1}f_*(c)}$ is $0$ as $f_*$ commutes with $dd^c$. We denote by $\mathcal{W}_c$ the superpotential of $d_p^{-1}f_*(\beta_c)-\beta_{d_p^{-1}f_*(c)}$ for each $c\in H^-$.
\medskip

We will use the following throughout the paper:
\begin{align*}
	c_n&:=d_p^{-n}(f^n)_*(c_0),\quad\mathcal{W}_n:=\mathcal{W}_{c_n},\\
	\beta_n&:=\textrm{ the unique form in } \Hc^- \textrm{ such that }\{\beta_n\}=c_n,\\
	\Omega_n&:=d_p^{-n}(f^n)_*(\Omega_0)\textrm{ and }\\
	\Uc_n'&:= \textrm{ the superpotential of }\Omega_n-\beta_n.	
\end{align*}

By the same arguments as in \cite[Proposition 4.4]{DNV}, $\mathcal{W}_c, \mathcal{W}_n, \Uc_n'$ are all continuous. Since $\Omega_0=\beta_0$, we have $\Uc_0'=0$. The spectral radius of $d_p^{-1}f_*$ on $H^-$ is smaller than $1$. Thus, once we fix a norm  $\|\cdot\|$ on $H^{k-p, k-p}(X, \R)$, we have $\|c_n\|=\Oc(\delta^n)$ and $\|\beta_n\|_\infty=\Oc(\delta^n)$ for some $0<\delta<1$.
\medskip

In the rest of the paper, we use $p$ to mean the order of the main dynamical degree $d_p$ for a surjective holomorphic endomorphism $f:X\to X$ with simple action on cohomology. For a general bidegree, we will use $s$.


\section{Equidistribution for currents with bounded superpotentials}

In this section, we prove that for a surjective holomorphic endomorphism $f:X\to X$ with simple action on cohomology and for a currrent $S\in\Cc_p$ with bounded superpotentials, Theorem \ref{thm:1st-main} is true. We start by proving the following lemma.
\begin{lemma}\label{lem:smooth_conv}
	Let $X$ and $f$ be as in Theorem \ref{thm:1st-main}. Let $\alpha$ be a closed smooth form of bidegree $(p, p)$. Then, we have
	\begin{align*}
		d_p^{-n}(f^n)^*(\alpha) \to cT_p^+
	\end{align*}
	where $T_p^+=\lim_{n\to\infty}d_p^{-n}(f^n)^*(\omega^p)$ and $c=\langle \alpha, T_p^-\rangle/\langle \omega^{p}, T_p^-\rangle$.
\end{lemma}

\begin{proof}
	The proof of {\cite[Lemma 5.2]{DNV}} implies that for any smooth closed positive $(p, p)$-current $\theta\ge \omega^p$, the sequence $d_p^{-n}(f^n)^*(\theta)$ converges to a closed current $T_\theta$ such that $d_p^{-1}f^*(T_\theta)=T_\theta$ in the sense of currents, its cohomology class $\{T_\theta\}\in L^+$, its superpotentials are continuous. Since $\alpha$ can be written as a difference of such two smooth closed positive $(p, p)$-currents, we see that the sequence $d_p^{-n}(f^n)^*(\alpha)$ converges in the sense of currents. Let $T_\alpha$ denote the limit current.
	\medskip
	
	Since the dimension of $L^+$ is $1$, there exists a constant $c>0$ such that the class of $T_\alpha - cT^+_p$ is $0$. Since $T_\alpha$ and $T^+_p$ are invariant under the action of $d_pf^*$, we have
	\begin{align*}
		\langle T_\alpha - cT^+_p, \varphi\rangle &= d_p^{-n}\langle (f^n)^*(T_\alpha-cT^+_p), \varphi\rangle=d_p^{-n}\Uc_{(f^n)^*(T_\alpha-cT^+_p)}(dd^c\varphi)\\
		&=d_p^{-n}\Uc_{T_\alpha-cT^+_p}((f^n)_*(dd^c\varphi))
	\end{align*}
	for every $n\in\N$ where $\varphi$ is a smooth test form of bidegree $(k-p, k-p)$. The last equality comes from Proposition \ref{prop:img_conti_suppot}. Since $T_\alpha$ and $T^+_p$ admit continuous superpotentials and therefore, their superpotentials are bounded on a $*$-bounded set. Since the dynamical degree $d_{p-1}$ of order ${p-1}$ of $f$ is strictly smaller than $d_p$, there exists $d_{p-1}<d'<d_p$ such that for all sufficiently large $n\in\N$, $\|(f^n)_*(dd^c\varphi)\|_*\lesssim (d')^n$ where the inequality is up to a constant multiple independent of $n\in\N$. So, we get
	\begin{displaymath}
		|\langle T_\alpha - cT^+_p, \varphi\rangle|\lesssim \lim_{n\to\infty}(d'/d_p)^n=0.
	\end{displaymath}
	Hence, we conclude that $T_\alpha=cT^+_p$.
	\medskip
	
	To find the constant $c>0$, it is enough to compare the mean of the currents.
	\begin{align*}
		c&=\dfrac{\langle T_\alpha, \omega^{k-p}\rangle}{\langle T^+_p, \omega^{k-p}\rangle}=\dfrac{\lim_{n\to\infty}\langle d_p^{-n}(f^n)^*(\alpha), \omega^{k-p}\rangle}{\lim_{n\to\infty}\langle d_p^{-n}(f^n)^*(\omega^p), \omega^{k-p}\rangle}\\
		&=\dfrac{\lim_{n\to\infty}\langle \alpha, d_p^{-n}(f^n)_*(\omega^{k-p})\rangle}{\lim_{n\to\infty}\langle \omega^p, d_p^{-n}(f^n)_*(\omega^{k-p})\rangle}=\dfrac{\langle \alpha, T^-_p\rangle}{\langle \omega^{p}, T^-_p\rangle}.
	\end{align*}
\end{proof}

\begin{theorem}\label{thm:equi_bdd}
	Let $X$ and $f$ be as in Theorem \ref{thm:1st-main}. Let $S$ be a current in $\Dc_p$ and $\alpha_S$ a smooth closed $(p, p)$-form cohomologous to $S$. Assume that the superpotential of $S-\alpha_S$ admits bounded superpotentials such that
	\begin{align*}
		|\Uc_{S-\alpha_S}(R)|\le M\|R\|_*.
	\end{align*}
	for $R\in\widetilde{\Dc}^0_{k-p+1}$. Then, $d_p^{-n}(f^n)^*(S)$ converges to $cT^+$ exponentially fast in the sense of currents where $c={\langle \alpha_S, T^-_p\rangle}/{\langle\omega^p, T^-_p\rangle}={(\{S\}\smile\{T^-_p\})}/{ (\{\omega^{p}\}\smile \{T^-_p\}) }$. More precisely, for a smooth test form $\varphi$ of bidegree $(k-p,k-p)$, we have
	\begin{align*}
		|\langle d_p^{-n}(f^n)^*(S)-cT^+_p, \varphi\rangle|\le C \|\varphi\|_{C^2} M \delta^n
	\end{align*}
	where $0<\delta<1$ and $C>0$ are constants independent of $M$, $S$, $f$ and $n$, and $T^+_p=\lim_{n\to\infty}d_p^{-n}(f^n)^*\omega^p$.
\end{theorem}

\begin{proof}[Proof of Theorem \ref{thm:equi_bdd}] Actually, the proof is essentially the same as that of Lemma \ref{lem:smooth_conv}. Lemma \ref{lem:smooth_conv} implies that $d_p^{-n}(f^n)^*(\alpha_S)$ converges to $cT^+_p$ where $c={\langle \alpha, T^-_p\rangle}/{\langle \omega^{k-p}, T^-_p\rangle}$.	For a smooth test form $\varphi$ of bidegree $(k-p, k-p)$, we have
	\begin{align*}
		\langle d_p^{-n}(f^n)^*(S)-cT^+_p, \varphi\rangle&=\langle d_p^{-n}(f^n)^*(S-\alpha_S), \varphi\rangle+\langle d_p^{-n}(f^n)^*(\alpha_S)-cT^+_p, \varphi\rangle\\
		&=d_p^{-n}\langle (f^n)^*(S-\alpha_S), \varphi\rangle+\langle d_p^{-n}(f^n)^*(\alpha_S-cT^+_p), \varphi\rangle.
	\end{align*}
	We consider the first term. Since $\varphi$ is smooth, it can be expressed as
	\begin{align*}
		d_p^{-n}\langle (f^n)^*(S-\alpha_S), \varphi\rangle=d_p^{-n}\Uc_{(f^n)^*(S-\alpha_S)}(dd^c\varphi).
	\end{align*}
	By Proposition \ref{prop:ext_to_conti}, we have
	\begin{align*}
		|d_p^{-n}\langle (f^n)^*(S-\alpha_S), \varphi\rangle|\le d_p^{-n}c_1M\|(f^n)_*(dd^c\varphi)\|_*\le d_p^{-n}c_2M\|\varphi\|_{C^2}\|(f^n)_*(\omega^{k-p+1})\|_*
	\end{align*}
	for some $c_1, c_2>0$. Since the dynamical degree $d_{p-1}$ of order ${p-1}$ of $f$ is strictly smaller than $d_p$, there exists $d_{p-1}<d'<d_p$ such that for all sufficiently large $n\in\N$, $\|(f^n)_*(\omega^{k-p+1})\|_*\le (d')^n$. Hence, we have
	\begin{align*}
		|d_p^{-n}\langle (f^n)^*(S-\alpha_S), \varphi\rangle|\le c_2M\|\varphi\|_{C^2}\left(\frac{d'}{d_p}\right)^n
	\end{align*}
	for all sufficiently large $n\in\N$.
	\medskip
	
	Concerning the second term, let $\beta_S$ be a smooth closed $(p, p)$-form cohomologous to $cT^+_p$. The current $cT^+_p-\beta_S$ admits continuous superpotentials. Applying the same argument as previously, we obtain that there exist constants $d_{p-1}<d''<d_p$ and $M_{T^+_p}>0$ such that 
	$$|d_p^{-n}\langle (f^n)^*(cT^+_p-\beta_S), \varphi\rangle|\le c_2M_{T^+_p}\|\varphi\|_{C^2}\left(\frac{d''}{d_p}\right)^n$$
	is true for all sufficiently large $n\in\N$. Also, $\alpha_S-\beta_S$ belongs to a class in $H^+$ in Lemma \ref{lem:decomposition_cohom}. Then, according to the proof of \cite[Lemma 5.7]{DNV}, the superpotential $\Uc_n$ of $d^{-n}(f^n)^*(\alpha_S-\beta_S)$ satisfies
	\begin{displaymath}
		|\langle d_p^{-n}\langle (f^n)^*(\alpha_S-\beta_S), \varphi \rangle|=|\Uc_n(dd^c\varphi)|\le nA\sigma^n\|dd^c\varphi\|_*
	\end{displaymath}
	for some constants $A>0$ and $0<\sigma<1$. Hence, putting the three estimates in the above altogether, we get the desired convergence with the desired speed of convergence.
\end{proof}

\section{Non-pluripolar products of currents}\label{sec:non-pluripolar}

In this section, we briefly recall the notion of the non-pluripolar product. For the detail of non-pluripolar products, we refer the reader to \cite{BEGZ}.

Locally, the non-pluripolar product is defined as below. Let $u_1, \cdots, u_s$ be psh functions on the complex manifold $X$. We define
$$ O_k:=\bigcap_{j=1}^s\{u_j>-k\}.$$

\begin{definition}
	If $u_1, \cdots, u_s$ are psh functions on the complex manifold $X$, we shall say that the non-pluripolar product $\langle \bigwedge_{j=1}^s dd^c u_j\rangle$ is well defined on $X$ if for each compact subset $K$ of $X$, we have
	\begin{align*}
		\sup_k\int_{K\cap O_k}\omega^{n-s}\wedge \bigwedge_{j=1}^sdd^c\max\{u_j, -k\}<\infty.
	\end{align*}
	Then, the non-pluripolar product $\langle \bigwedge_{j=1}^s dd^c u_j\rangle$ is defined to be a current satisfying the following:
	\begin{enumerate}
		\item $\mathbf{1}_{O_k}\langle \bigwedge_{j=1}^s dd^c u_j\rangle=\mathbf{1}_{O_k}\bigwedge_{j=1}^s dd^c \max\{u_j, -k\}$;
		\item the current $\langle \bigwedge_{j=1}^s dd^c u_j\rangle$ has no mass on the pluripolar set
		$$ X\setminus \bigcup_{k=1}^\infty O_k=\{u=-\infty\}.$$
	\end{enumerate}
\end{definition}
Here, $\omega$ is a strictly positive $(1, 1)$-form on $X$ with respect to which the mass $\|\cdot\|$ of a current is measured. The condition in the above is independent of the choice of $\omega$. 

An important property is that the non-pluripolar product of globally defined currents on compact K\"ahler manifolds is always well defined.

\begin{theorem}[Proposition 1.6 and Theorem 1.8 in \cite{BEGZ}]
	Let $R_1, \cdots, R_s$ be closed positive $(1, 1)$-currents on a compact K\"ahler manifold $X$. Then their non-pluripolar product $\langle R_1\wedge \cdots \wedge R_s\rangle$ is a well-defined closed positive $(s, s)$-current.
\end{theorem}

\section{Equidistribution for non-pluripolar product currents}

We prove our main theorem in several steps. Our standing assumption throughout this section is that $X$ and $f$ are as in Theorem \ref{thm:1st-main}. First, we start with the following observation.
\begin{proposition}\label{pro-boundedpotential} 
	Let $\theta_j$ be K\"ahler forms and $u_j$ bounded $L^1$ functions on $X$ such that $\theta_j+dd^cu_j\ge 0$ for $j=1, \cdots, p$. Then, $$d_p^{-n}(f^n)^* \left(\bigwedge_{j=1}^p(\theta_j+dd^cu_j)\right)\to c T^+_p,\quad c=\dfrac{\langle \theta_1\wedge\cdots\wedge\theta_p, T^-_p\rangle}{\langle \omega^{k-p}, T^-_p\rangle}$$
	exponentially fast in the sense of currents.
\end{proposition}

By Theorem \ref{thm:equi_bdd}, it is enough to observe that $\bigwedge_{j=1}^p(\theta_j+dd^cu_j)-\bigwedge_{j=1}^p\theta_j$ admits bounded superpotentials, as we will see below:

	\begin{proposition}\label{prop:bounded_(1,1)}
	There exists a constant $C>0$ independent of the family $(u_j)_{j=1, \cdots, p}$ such that for $R\in\widetilde{\Dc}^0_{k-p+1}$
	\begin{displaymath}
		\left|\Uc_{(\bigwedge_{j=1}^p(\theta_j+dd^cu_j)-\bigwedge_{j=1}^p\theta_j)}(R)\right|\le C\left(\sup_j\|u_j\|_\infty\right)\|R\|_*.
	\end{displaymath}
\end{proposition}

\begin{proof}
	We have
	\begin{align*}
		\bigwedge_{j=1}^p(\theta_j+dd^cu_j)-\bigwedge_{j=1}^p\theta_j=dd^c\sum_{i=1}^{p}u_i\left(\bigwedge_{j=1}^{i-1}\theta_j\right)\wedge \left(\bigwedge_{j=i+1}^{p}(\theta_j+dd^cu_j)\right).
	\end{align*}
	Hence, $\sum_{i=1}^{p}u_i\left(\bigwedge_{j=1}^{i-1}\theta_j\right)\wedge \left(\bigwedge_{j=i+1}^{p}(\theta_j+dd^cu_j)\right)$ is a potential of $\bigwedge_{j=1}^p(\theta_j+dd^cu_j)-\bigwedge_{j=1}^p\theta_j$. Direct computations with cohomological arguments complete the proof.
\end{proof}

In the second step, we approximate the non-pluripolar product by currents of the form in Proposition \ref{pro-boundedpotential} outside of the pluripolar set associated with the currents $R_1, \cdots, R_p$. Indeed, Lemma \ref{le-psinho} implies that our approximation is valid as the non-pluripolar product does not charge any mass on the pluripolar set. We first prove two preparatory lemma and proposition.

\begin{lemma} \label{le-lemmaFhobatbien} Let $\mathcal{F}$ be a set of closed positive $(p,p)$-currents $S$ such that $S \le T^+_p$ and $\{S\} \in L^+$ (for $L^+$, see Section \ref{sec:simple_action}). Assume that for each $S\in \mathcal{F}$, there exists $S'\in\mathcal{F}$ such that $d_p^{-1}f^*(S')= S$. Then for every $S \in \mathcal{F}$ there exists a constant $\lambda_S \ge 0$ such that $S= \lambda_S T^+_p$. 
\end{lemma}

\proof The proof was implicitly contained in the proof of \cite[Proposition 5.9]{DNV}. We recall it for reader's convenience. 
\medskip

Let $S \in \mathcal{F}$ be a current. Then, $\{ S\} \in L^+$ and $\dim L^+=1$ imply that $\{S\}= \lambda_S  \{T^+_p\}$ for some constant $\lambda_S\in\R$. From the hypothesis on $\mathcal{F}$, there exists a sequence $(S_m)_{m\in\N}$ of currents in $\mathcal{F}$ such that 
$$S- \lambda_S T^+_p = d_p^{-m} (f^m)^*(S_m - \lambda_S T^+_p),$$
for every $m \in \N$. Hence 
$$\mathcal{U}_{S- \lambda_S T^+_p}(R)= \mathcal{U}_{S_m- \lambda_S T^+_p}\big(d_p^{-m} (f^m)_*(R) \big)$$
for every $R \in \widetilde{\Dc}^0_{k-p+1}(X)$. Recall that 
$$\|d_p^{-m} (f^m)_*(R)\|_* \le \delta^m,$$
for some constant $\delta \in (0,1)$ independent of $m$ and $R$. Observe that $S_m$ and $\lambda_ST^+_p$ belong to the same cohomology class. Since for every $m\in\N$, $0\le S_m\le T^+_p$ and $T^+_p$ admits a bounded superpotential, \cite[Remark 2.6]{DNV} implies that
$$ |\mathcal{U}_{S- \lambda_S T^+_p}(R)|=|\Uc_{S_m-\lambda_ST^+_p}(d_p^{-m}(f^m)_*(R))|\le c\delta^m $$
for some constant $c>0$. Letting $m \to \infty$, we infer that $\Uc_{S- \lambda_S T^+_p}=0$ on $\widetilde{\Dc}^0_{k-p+1}(X)$. In other words, $S= \lambda_S T^+_p$. This finishes the proof.
\endproof

\begin{proposition} \label{pro-SphayS} Let $S\in\Cc_p$ be a current. Assume that $d_p^{-n} (f^n)^* S \to T^+_p$ as $n \to \infty$. Let $S'$ be a closed positive $(p, p)$-current with $S' \le S$. Let $\psi$ be a quasi-psh function such that $0\le \psi \le 1$.  Then, we have $d_p^{-n} (f^n)^* (\psi S') \to c' T^+_p$ where $c'= \langle T^-_p, \psi S' \rangle / \langle T^-_p, \omega^{p} \rangle $. 
\end{proposition}

Here, the value $\langle T^-_p, \psi S' \rangle$ makes sense since $T_p^-$ has continuous superpotentials.

\proof Observe first that the currents $d_p^{-n} (f^n)^* (\psi S')$ have uniformly bounded mass independent of $n\in\N$. Let $\mathcal{F}$ be the set of limit currents of the family $(d_p^{-n} (f^n)^* (\psi S'))_{n\in\N}$. Clearly, every current in $\mathcal{F}$ is positive. Let $R$ be a current in $\mathcal{F}$ and $(n_j)_{j\in\N}$ a sequence such that $R= \lim_{j \to \infty} d_p^{-n_j} (f^{n_j})^* (\psi S')$. Let $R'$ be a limit current of the sequence $$(d_p^{-n_j+1} (f^{n_j-1})^* (\psi S'))_{j \in \N}.$$
Then, $d_p^{-1}f^*(R')$ is a limit current of the sequence 
$(d_p^{-n_j} (f^{n_j})^* (\psi S'))_{j\in\N}$. It follows that $R= d_p^{-1} f^* (R')$. In other words, for every $R\in\mathcal{F}$, there exists $R'\in\mathcal{F}$ such that $d_p^{-1}f^*(R')= R$.
\medskip

Since $dR= \lim_{j \to \infty} d_p^{-n_j} (f^{n_j})^* (d\psi \wedge  S')$, arguing as in the proof of \cite[Lemma 5.8]{DNV}, one gets $d R= 0$ and therefore, $R$ is closed. Using the fact that $S' \le S$ and $0\le \psi\le 1$, we infer $R \le T^+_p$.
\medskip

We check that $\{R\}\in L^+$. Let $(\theta_i)$ be a finite set of smooth closed $(k-p, k-p)$-currents such that their de Rham cohomology classes $\{\theta_i\}$'s form a basis of $H^{k-p, k-p}(X, \R)$. By replace $(\theta_i)$ by $(M\omega^{k-p}+\theta_i)$ for a sufficiently large $M>0$, we may suppose that all $\theta_i$'s are closed positive. Let $(\theta_i^*)$ be another set of smooth closed $(p, p)$-currents such that their cohomology classes $\{\theta^*_i\}$'s form a dual basis of $(\{\theta_i\})$ in the Poincar\'e duality theorem. By the Poincar\'e duality theorem, each element $F\in \mathcal{F}$ is cohomologous to a linear combination $\sum_i c_{F,i}\theta^*_i$. As shown just before, we have $F\le T^+_p$ for every $F\in\mathcal{F}$. Then, due to the positivity of $F\in\mathcal{F}$, $T^+_p$, and all $\theta_i$'s, we have 
$$0\le c_{F, i}\le c_{T^+_p, i}\quad\textrm{ for all }i \textrm{ and for all }F\in\mathcal{F}.$$
In particular, for every $F\in\mathcal{F}$, its de Rham cohomology class can be written as
\begin{align}
	\label{eq:deRham_coeff}	\{F\}=\sum c_{F, i}\{\theta^*_i\},\quad 0\le c_{F, i}\le c_{T^+, i}\quad\textrm{ for all }i.
\end{align}

Now, let $\alpha:= \{R\}$ be the de Rham cohomology class of $R$ and $\mathcal{F}'$ the image of $\mathcal{F}$ in $H^{p,p}(X,\R)$. By the surjectivity of $d^{-1}f^*$ on $\mathcal{F}'$ proven in the above, for each $m\in\N$, we can write $\alpha= d_p^{-m} (f^m)^* \alpha_m$ for some $\alpha_m \in \mathcal{F}'$. Together with the boundedness of coefficients in \eqref{eq:deRham_coeff}, \cite[Lemma 5.2]{DNV} implies $\alpha \in L^+$. 
\medskip

Finally, by Lemma \ref{le-lemmaFhobatbien} applied to $\mathcal{F}$, we obtain that $R= \lambda_R T^+_p$ for some constant $\lambda_R>0$. Furthermore, since
\begin{align*}
\int_X R \wedge \omega^{k-p}&= \lim_{j \to \infty} d_p^{-n_j} \int_X (f^{n_j})^* (\psi S') \wedge \omega^{k-p}= \lim_{j \to \infty} d_p^{-n_j} \int_X \psi S' \wedge (f^{n_j})_* (\omega^{k-p})\\
&= \int_X \psi S' \wedge T^{-}_p,
\end{align*}
we get $\lambda_R= \langle T^-_p, \psi S' \rangle / \langle T^+_p, \omega^{k-p} \rangle =\langle T^-_p, \psi S' \rangle / \langle T^-_p, \omega^{p} \rangle$ which is independent of $R$. Here, in the third and last integrals, we abused notation. Indeed, since $T^-_p$ and $(f^n)_*\omega^{k-p}$ admit continuous superpotentials, the third and last integrals makes sense. Hence, $\mathcal{F}$ consists of a single element. This finishes the proof.
\endproof

\begin{lemma}\label{le-psinho} Let $\phi$ be a negative quasi-psh function on $X$. Let $\phi_m:= m^{-1}\max\{\phi, -m\}$. Let $S$ be a closed positive $(p,p)$-current having no mass on $\{\phi=-\infty\}$. For each $m\in\N$, let $R_m$ be a limit current of the sequence $(d_p^{-n} (f^n)^*(\phi_m S))_{n\in\N}$ as $n \to \infty$. Then, one has $\lim_{m \to \infty} R_m=0$.
\end{lemma}

\proof
Recall that $c^-$ is the class of $T^-_p$. Let $\alpha_0$ be a form in $c^-$. Put $\alpha_n:= d_p^{-n} (f^n)_* \alpha_0$ which is a form in $c^-$. By the same arguments in \cite[Lemma 5.6]{DNV}, the superpotential $\Vc_n$ of $\alpha_n - \alpha_0$ converges to a continuous functional on $\mathcal{D}^0_{p-1}(X)$ uniformly on a set of bounded $*$-norm. In particular, there exists a constant $C>0$ such that
\begin{align}\label{ine-tinhcahtVngan}
	|\Vc_n(R)| \le C \|R\|_*
\end{align}
holds for every $n\in\N$ and for every $R \in \mathcal{D}^0_{p-1}$.
\medskip

Notice that $d_p^{-n}(f^n)_*\omega^{k-p}- (\alpha_0 + \beta_n) $ is cohomologous to $0$, and its superpotential $\Uc_n$ satisfies
$$\Uc_n= \Vc_n+ \Uc'_n,$$
where $\Uc'_n$ converges to $0$ uniformly on a set of $*$-bounded norm in $\mathcal{D}^0_{p-1}$.
\\

Now, we show $\lim_{m\to\infty}R_m=0$. By passing to a subsequence, we may suppose that $d_p^{-n} (f^n)^*(\phi_m S) \to R_m$ as $n \to \infty$. Since $(f^n)_*(\omega^{k-p})$ admits continuous superpotentials, by approximating $\phi_m$ and $S$ with smooth ones, we can write
\begin{align*}
	\int_X R_m \wedge \omega^{k-p}&= \lim_{n \to \infty} d_p^{-n} \int_X (f^{n})^* (\phi_m S) \wedge \omega^{k-p}\\
	&= \lim_{n\to\infty}\left[\int_X \phi_m S \wedge (\alpha_0 + \beta_n)+ \Uc_n(\ddc \phi_m \wedge S)\right]\\
	&= \lim_{n\to\infty}\left[\int_X \phi_m S \wedge (\alpha_0 + \beta_n)+ \Vc_n (\ddc \phi_m \wedge S)+ \Uc_n' ( \ddc \phi_m \wedge S)\right].
\end{align*}
{The first term converges to $0$ as $m \to \infty$ because the uniform norm of $\beta_n$ converges to $0$ as $n\to\infty$ and $S$ has no mass on $\{\phi=-\infty\}$.} Concerning the second term, since $\phi$ is $M\omega$-psh for $M$ big enough, one sees that 
$$\ddc \psi_m \wedge S= m^{-1} (\ddc \max\{\phi, -m\}+ M\omega) \wedge S- m^{-1} M \omega \wedge S$$
and that $\|\ddc \psi_m \wedge S\|_*\lesssim m^{-1}$ where the inequality is up to a constant multiple independent of $n$ and $m$. Together with this $*$-norm estimate, \eqref{ine-tinhcahtVngan} yields
$$|\Vc_n(\ddc \psi_m \wedge S)| \lesssim \|\ddc \psi_m \wedge S\|_* \lesssim m^{-1}$$
where the inequality is up to a constant multiple independent of $n$ and $m$. The third term converges to $0$ by properties of $\mathcal{U}'_n$ since the set $(dd^c\psi_m\wedge S)_{m\in\N}$ is of bounded $*$-mass. This completes the proof.    
\endproof

\begin{proof}[Proof of Theorem \ref{thm:1st-main}] Write $R_j= \ddc u_j + \theta_j$. We may assume that $u_j\le -1$. Let $M>0$ be a big enough constant such that $\theta_j + M \omega$ is K\"ahler for every $j$. By Proposition \ref{pro-SphayS}, it suffices to prove the desired assertion for $R_j+ M \omega$ in place of $R_j$. In other words, one can assume that $\theta_j$ is K\"ahler. Let $\phi:= u_1+ \ldots + u_p$ and $S:= \langle R_1 \wedge \cdots \wedge R_p \rangle$. Then, $S$ has no mass on the pluripolar set $\{\phi = -\infty\}$. Let $\phi_m:= m^{-1} \max\{\phi, -m\}$. We decompose 
	$$S= (1+ \phi_m)S + (-\phi_m)S.$$
	By Lemma \ref{le-psinho}, the mass of any limit currents of $d_p^{-n} (f^n)^* (-\phi_m S)$ converges uniformly  to $0$ as $m \to \infty$.
\medskip

Now observe that 
	$1+ \phi_m =0$ on $\cup_j \{u_j \le -m\}$. By the pluri-locality of Monge-Amp\`ere operators, one obtains 
	$$(1+\phi_m)S= (1+ \phi_m) \bigwedge_j (\ddc \max\{u_j, -m\}+ \theta_j).$$

Since $\bigwedge_j (\ddc \max\{u_j, -m\}+ \theta_j)$ admits a bounded superpotential, from Theorem \ref{thm:equi_bdd} we get that it converges to $cT^+_p$ where $c>0$ is a constant independent of $m$. So, Proposition \ref{pro-SphayS} implies that 
	$$d_p^{-n} (f^n)^* ((1+\phi_m) S) \to \lambda_m T^+_p,\quad \lambda_m=\langle T^-_p, (1+\phi_m) S \rangle / \langle T^-_p, \omega^{p} \rangle.$$
	
	If $R$ is a limit current of $d_p^{-n} (f^n)^* S$, we have
	$$R= \lambda_m T^+_p + \epsilon(m)\quad \textrm{and}\quad \lim_{m\to\infty}\epsilon(m)=0$$
	for every $m\in\N$. As we can see, $\lambda:=\lim_{m\to\infty}\lambda_m$ exists by the monotone convergence theorem. Thus, we get $R= \lambda T^+_p$. The independence of $\lambda$ on $R$ implies that there exists only one limit current of the family $(d^{-n} (f^n)^* S)_{n\in\N}$. So, the theorem is proved. 
\end{proof}

\begin{remark}
	Our method works for a holomorphic correspondence $f$ such that $f^{-1}$ is also a holomorphic correspondence. But for claity, we present our method for surjective holomorphic endomorphisms.
\end{remark}


\noindent
\Addresses
\end{document}